\documentclass[12pt]{article}

\usepackage{amsmath,amssymb,amsthm}
\newtheorem{theorem}{Theorem}

\def\0{\leqno}

\title{Finite groups with a certain number\\ of cyclic subgroups II}
\author{Marius T\u arn\u auceanu}
\date{May 3, 2016}

\begin{document}
\maketitle

\begin{abstract}
    In this note we describe the finite groups $G$ having $|G|-2$ cyclic subgroups. This partially solves the open problem in the end of \cite{3}.
\end{abstract}



\bigskip

Let $G$ be a finite group and $C(G)$ be the poset of cyclic subgroups of $G$. The connections between $|C(G)|$ and $|G|$ lead to characterizations of certain finite groups $G$. For example, a basic result of group theory states that $|C(G)|=|G|$ if and only if $G$ is an elementary abelian $2$-group. Recall also the main theorem of \cite{3}, which states that $|C(G)|=|G|-1$ if and only if $G$ is one of the following groups: $\mathbb{Z}_3$, $\mathbb{Z}_4$, $S_3$ or $D_8$.

In what follows we shall continue this study by describing the finite groups
$G$ for which $$|C(G)|=|G|-2.\0(*)$$First, we observe that
certain finite groups of small orders, such as $\mathbb{Z}_6$,
$\mathbb{Z}_2\times\mathbb{Z}_4$, $D_{12}$ and $\mathbb{Z}_2\times D_8$, have this property. Our main
theorem proves that in fact these groups exhaust all finite groups $G$ satisfying $(*)$.

\begin{theorem}
    Let $G$ be a finite group. Then $|C(G)|=|G|-2$ if and only if $G$ is one of the following groups:
    $\mathbb{Z}_6$, $\mathbb{Z}_2\times\mathbb{Z}_4$, $D_{12}$ or $\mathbb{Z}_2\times D_8$.
\end{theorem}

\begin{proof} We will use the same technique as in the proof of Theorem 2 in \cite{3}.
Assume that $G$ satisfies $(*)$, let $n=|G|$ and denote by
$d_1=1,d_2,...,d_k$ the positive divisors of $n$. If $n_i=|\{H\in C(G)\mid
|H|=d_i\}|$, $i=1,2,...,k$, then
$$\sum_{i=1}^k n_i\phi(d_i)=n.$$Since $|C(G)|=\sum_{i=1}^k
n_i=n-2$, one obtains
$$\sum_{i=1}^k n_i(\phi(d_i)-1)=2,$$which implies that we have the following possibilities:

\bigskip\hspace{10mm} Case 1. There exists $i_0\in\{1,2,...,k\}$ such that $n_{i_0}(\phi(d_{i_0})-1)=2$ and $n_i(\phi(d_i)-1)=0, \forall\, i\neq i_0$. \medskip

Since the image of the Euler's totient function does not contain odd integers $>1$, we infer that $n_{i_0}=2$ and $\phi(d_{i_0})=2$, i.e. $d_{i_0}\in\{3,4,6\}$. We remark that $d_{i_0}$ cannot be equal to 6 because in this case $G$ would also have a cyclic subgroup of order 3, a contradiction. Also, we cannot have $d_{i_0}=3$ because in this case $G$ would contains two cyclic subgroups of order 3, contradicting the fact that the number of subgroups of a prime order $p$ in $G$ is $\equiv 1\, ({\rm mod}\, p)$ (see e.g. the note after Problem 1C.8 in \cite{1}). Therefore $d_{i_0}=4$, i.e. $G$ is a 2-group containing exactly two cyclic subgroups of order 4. Let $n=2^m$ with $m\geq 3$. If $m=3$ we can easily check that the unique group $G$ satisfying $(*)$ is $\mathbb{Z}_2\times\mathbb{Z}_4$. If $m\geq 4$ by Proposition 1.4 and Theorems 5.1 and 5.2 of \cite{2} we infer that $G$ is isomorphic to one of the following groups:
\begin{itemize}
\item[-] $M_{2^m}$;
\item[-] $\mathbb{Z}_2\times\mathbb{Z}_{2^{m-1}}$;
\item[-] $\langle a,b \,|\, a^{2^{m-2}}=b^8=1,\, a^b=a^{-1},\, a^{2^{m-3}}=b^4\rangle$, where $m\geq 5$;
\item[-] $\mathbb{Z}_2\times D_{2^{m-1}}$;
\item[-] $\langle a,b \,|\, a^{2^{m-2}}=b^2=1,\, a^b=a^{-1+2^{m-4}}c,\, c^2=[c,b]=1,\, a^c=a^{1+2^{m-3}}\rangle$, where $m\geq 5$.
\end{itemize}All these groups have cyclic subgroups of order 8 for $m\geq 5$ and thus they does not satisfy $(*)$. Consequently, $m=4$ and the unique group with the desired property is $\mathbb{Z}_2\times D_8$.

\bigskip\hspace{10mm} Case 2. There exist $i_1, i_2\in\{1,2,...,k\}$, $i_1\neq i_2$, such that $n_{i_1}(\phi(d_{i_1})-1)=n_{i_2}(\phi(d_{i_2})-1)=1$ and $n_i(\phi(d_i)-1)=0, \forall\, i\neq i_1, i_2$. \medskip

Then $n_{i_1}=n_{i_2}=1$ and $\phi(d_{i_1})=\phi(d_{i_2})=2$, i.e. $d_{i_1},d_{i_2}\in\{3,4,6\}$. Assume that $d_{i_1}<d_{i_2}$. If $d_{i_2}=4$, then $d_{i_1}=3$, that is $G$ contains normal cyclic subgroups of orders 3 and 4. We infer that $G$ also contains a cyclic subgroup of order 12, a contradiction. If $d_{i_2}=6$, then we necessarily must have $d_{i_1}=3$. Since $G$ has a unique subgroup of order 3, it follows that a Sylow 3-subgroup of $G$ must be cyclic and therefore of order 3. Let $n=3\cdot 2^m$, where $m\geq 1$. Denote by $n_2$ the number of Sylow 2-subgroups of $G$ and let $H$ be such a subgroup. Then $H$ is elementary abelian because $G$ does not have cyclic subgroups of order $2^i$ with $i\geq 2$. By Sylow's Theorems, $$n_2|3 \mbox{ and } n_2\equiv 1\, ({\rm mod}\, 2),$$implying that either $n_2=1$ or $n_2=3$. If $n_2=1$, then $G\cong \mathbb{Z}_{2}^m\times\mathbb{Z}_3$, a group that satisfies $(*)$ if and only if $m=1$, i.e. $G\cong \mathbb{Z}_6$. If $n_2=3$, then $|{\rm Core}_G(H)|=2^{m-1}$ because $G/{\rm Core}_G(H)$ can be embedded in $S_3$. It follows that $G$ contains a subgroup isomorphic with $\mathbb{Z}_{2}^{m-1}\times\mathbb{Z}_3$. If $m\geq 3$ this has more than one cyclic subgroup of order 6, contradicting our assumption. Hence either $m=1$ or $m=2$. For $m=1$ one obtains $G\cong S_3$, a group that does not have cyclic subgroups of order 6, a contradiction, while for $m=2$ one obtains $G\cong D_{12}$, a group that satisfies $(*)$. This completes the proof.
\end{proof}

\vspace*{5ex}\small

\hfill
\begin{minipage}[t]{5cm}
Marius T\u arn\u auceanu \\
Faculty of  Mathematics \\
``Al.I. Cuza'' University \\
Ia\c si, Romania \\
e-mail: {\tt tarnauc@uaic.ro}
\end{minipage}

\end{document}